\documentclass[leqno, 11pt]{amsart}
\usepackage{amsmath}
\usepackage[sorted-cites]{amsrefs}
\usepackage{verbatim}
\usepackage{amscd}
\usepackage{amssymb}
\usepackage{amsmath, amssymb}
\usepackage{amsthm}
\usepackage[T1]{fontenc}
\usepackage[ansinew]{inputenc}
\usepackage{graphicx}

\newcommand{\Vol}{\operatorname{Vol}}

\newtheorem{theorem}{Theorem}[section]
\newtheorem{theorem/definition}{Theorem/Definition}[section]

\newtheorem{lemma}{Lemma}[section]

\theoremstyle{remark}
\newtheorem{remark}{Remark}[section]

\begin{document}
\title
{Volume estimate about shrinkers }
\author{ Xu Cheng  and Detang Zhou}

\address{Instituto de Matematica\\ Universidade Federal Fluminense\\
Niter\'oi, RJ 24020\\Brazil} \email{xcheng@impa.br}
\email{zhou@impa.br}

\begin{abstract}We  derive a precise estimate on the volume growth
of the level set of a  potential  function on a complete noncompact
Riemannian manifold. As   applications, we obtain the volume growth
rate of a complete noncompact self-shrinker and a
gradient shrinking Ricci soliton. We also  prove the equivalence of  weighted
volume finiteness, polynomial  volume growth and properness of an immersed
self-shrinker in Euclidean space.

\end{abstract}
\maketitle
\date{}


\footnotetext[1]{Both authors were partially supported by CNPq and
FAPERJ, Brazil.}

\section{Introduction}\label{Sec-1}
In this paper, we estimate  the volume  of complete  noncompact
Riemannian manifolds without assuming curvature conditions.
Let $(M, g)$ be an $n$-dimensional complete noncompact Riemannian
manifold  on which there exists a nonnegative smooth function $f$
satisfying ``shrinking soliton'' conditions. 

Let $\Delta_f$ be  the linear
operator defined  by $$\Delta_fu=\Delta u -\langle\nabla f,
\nabla u\rangle.$$ 
It is known that the operator $\Delta_f$ is
symmetric with respect to the weighted measure $e^{-f}dv$ on the space of  smooth function  with compact support.
Denote the level set  of $2\sqrt{f}$,  its volume and its weighted volume by
$$ D_r=\{x\in M: 2\sqrt{f}< r\} ,$$
$$ V(r)=\text{Vol}(D_r)=\int_{D_r}dv,$$
and $$ V_f(r)=\text{Vol}_f(D_r)=\int_{D_r}e^{-f}dv.$$
We get  the  volume growth estimate of the set $D_r$ and
the finiteness of the weighted volume of $M$ in the following
\begin{theorem}\label{thm-f-1} (Thm.\ref{thm-f}) Let $(M, g)$ be a complete noncompact Riemannain manifold.  If $f$ is a  proper $C^{\infty}$ function on $M$ satisfying $|\nabla f|^2\leq f$ on  $D_r$ for all $r$ and
\begin{equation} \Delta_f f+f
\leq k \label{I-0}
\end{equation}
for some constant $k$,
 then $M$ has finite weighted volume, that is,  
 \begin{equation*} V_f(M)=\int_M e^{-f} dv <+\infty
 \end{equation*} and
$$V(r)\leq C r^{2k}$$
for $r\geq 1$, where $C$ is a constant depending only on
$\int_M e^{-f} dv$.
\end{theorem}
The classical volume comparison theorems in Riemannian geometry are
one of  the basic ingredients of the analysis on manifolds. It is
usually  done by assuming that $M$ is an $n$-dimensional
Riemannian manifold with Ricci curvature being bounded below.  In our theorem, instead of assuming  curvature conditions,
we assume the existence of some functions which arise naturally in
the study of gradient shrinking Ricci solitons of Ricci flow and
self-shrinkers of mean curvature flow. Both  gradient shrinking
Ricci solitons and self-shrinkers are important in the singularity
analysis of
  Ricci flow and mean curvature flow respectively (\cite{MR1375255}, \cite{Hu}, \cite{P}) and recently  have been studied very much   (see  
  \cite{CM1}, \cite{CM2},  survey papers \cite{Cao09}, \cite{MR2648937},  etc). We observe both of them satisfy  inequality (\ref{I-0}) and so Theorem \ref{thm-f-1} can be applied  to them. 
Besides, Theorem \ref{thm-f-1} is of independent interest.  We only need to find a proper function on $M$ satisfying the
assumption of the  theorem.

As the first application,  we consider
self-shrinkers for mean curvature flow (MCF) in $\mathbb{R}^{n+1}$.
Recall that a one-parameter family of hypersurfaces $\Sigma_t\subset
\mathbb{R}^{n+1} $ flows by mean curvature if
$$\partial _tx=-H\bf{n},
$$
where $\bf{n}$ is the outward unit normal and $H$ is the mean
curvature. The flow $\Sigma_t$ is said to be a self-shrinker if
$$\Sigma_t=\sqrt{-t}\Sigma_{-1}
$$
for all $t<0$. This is equivalent to that $\Sigma=\Sigma_{-1}$
satisfies the equation
\begin{equation}H=\frac{\langle x, \bf{n}\rangle}{2}. \label{I-4}
\end{equation}
We also call hypersurface $\Sigma$ a self-shrinker in $\mathbb{R}^{n+1}$. We
obtain
\begin{theorem} \label{thm-S-2-1}  Let $\Sigma^n$ be a complete noncompact properly immersed self-shrinker  in Euclidian space $\mathbb{R}^{n+1}$ with nonnegative constant $\beta\leq\inf H^2$.  Then,
there is a positive constant $C$ such that for $r\geq 1$
$$ \Vol( B_{r}(0)\cap \Sigma)\leq C r^{n-2\beta}.$$
\end{theorem}
The upper bounds of the volume
 in Theorem
\ref{thm-S-2-1} are optimal because $\mathbb{R}^n$ and
cylinder self-shrinker  $S^{k}(\sqrt{2k})\times \mathbb{R}^{n-k}$  have the corresponding volume growth
rates.
 In  \cite{DX}, Ding and Xin proved the Euclidean volume growth of a
 complete noncompact properly immersed self-shrinker.  Theorem \ref{thm-S-2-1} implies their result.

 Further, in  Section \ref{S-5} we study the relation among  weighted
volume finiteness,  polynomial volume growth and
properness of a  self-shinker. In the works of Ecker-Huisken
\cite{MR1025164} and Colding-Minicozzi \cite{CM1} the polynomial
volume growth plays an important role in studying self-shrinker. It
is satisfied when the self-shrinker is a time-slice of a blow up. In
general, for a complete noncompact self-shrinker $\Sigma$, we show
in Theorem \ref{thm4-1} that the following statements are equivalent,

\begin{theorem}\label{thm4-1}  For any  complete immersed self-shrinker $\Sigma^n$ in $\mathbb{R}^{n+1}$, the following statements are equivalent,
\begin{enumerate}
\item  [(i)]  $\Sigma$ is proper;
\item  [(ii)] $\Sigma$ has Euclidean volume growth;
\item  [(iii)] $\Sigma$ has  polynomial volume growth;
\item  [(iv)] $\Sigma$ has  finite weighted volume, that is, $\int_\Sigma e^{-\frac{|x|^2}{4}}dv<\infty.$
\end{enumerate}
\end{theorem}

Theorem \ref{thm-S-2-1} and \ref{thm4-1} are the consequences of the corresponding general results of self-shrinkers in $\mathbb{R}^{n+p}, p\geq 1$  (see Theorem \ref{thm-S} and \ref{thm5}).

Using Theorem \ref{thm4-1}, we can rephrase the compactness theorem
of Colding-Minicozzi \cite{CM2} as the following

\begin{theorem}(Colding-Minicozzi \cite{CM2}) Given an integer $g\ge 0$ and a constant $A>0$, the set
\[\mathcal{ S}(g,A):=\left\{
\begin{split}\Sigma: \Sigma&\textrm{ is a complete smooth embedded self-shrinker in $\mathbb{R}^3$ with }\\
&\textrm{genus at most $g$ and }
\int_{\Sigma}e^{-\frac{|x|^2}{4}}dv\le A
\end{split} \right\}
\]
is compact with respect to the topology of $C^m$ convergence on
compact subsets for any $m\ge 2$.
\end{theorem}

The  second application of Theorem \ref{thm-f-1} is about gradient shrinking solitons. Recall
that a complete
 smooth Riemannian manifold $(M^n, g)$ is called a
{\it normalized gradient shrinking Ricci soliton} if there is a
smooth function $f$ on $M^n$ such that the Ricci tensor $R_{ij}$ of
the metric $g$ is given by
\begin{equation}R_{ij}+\nabla_i\nabla_jf=\frac 12 g_{ij}. \label{I-1}
\end{equation}
It is known that the scalar curvature $R$ of the metric $g$ is nonnegative \cite{MR2520796}. We get
\begin{theorem}\label{thm-shrinking-1} (Thm.\ref{thm-shrinking}) Let $(M^n, g, f)$ be a complete noncompact gradient shrinking solition with $f$ satisfying (\ref{I-1}) and $\beta=\inf R$. Then
 there exists  a  constant $C>0$ such that
$$ \Vol(B_{r}(x_0))\leq C r^{n-2\beta}$$
for $r\geq 1$, where  $C$ is a constant depending on $x_0$ and $ \Vol(B_{r}(x_0))$ denotes the volume of the geodesic ball of $M$ of radius $r$ centered at $x_0\in M$.
\end{theorem}

Since the scalar curvature $R$  is nonnegative (\cite{MR2520796}),
Theorem  \ref{thm-shrinking-1} implies the estimate of Euclidean
volume growth which has been obtained by Cao and the second author
in  \cite{MR2732975}.   The volume upper bounds of  the geodesic ball  $B_r(x_0)$ in Theorem
\ref{thm-shrinking-1} are sharp by the examples like  the Gaussian shrinker and
cylinder shrinking solitons.  After the first version was posted on arXiv, we were informed that Theorem 1.5 had been proved by Zhang \cite{Zh}. In \cite{Zh}, the author used the method  in \cite{MR2732975} to prove the result.

 It has been proved (Cao-Chen-Zhu \cite{MR2488948}), that {\sl
any 3-dimensional complete noncompact non-flat shrinking gradient
soliton is necessarily the round cylinder $\mathbb{S}^2\times
\mathbb{R}$ or one of its $\mathbb Z_2$ quotients}.  But for higher
dimensions, one has no classification of  complete noncompact
gradient shrinking solitons without any assumptions
on curvatures.


\section{Volume growth estimate}

In this section, we will prove Theorem \ref{thm-f-1}.   The first part of our argument is  similar to the proof of  Ding and Xin in  the case of self-shrinkers.
We have the following
\begin{theorem}\label{thm-f} (Thm.\ref{thm-f-1}) Let $(M, g)$ be a complete noncompact Riemannain manifold. Let $f$ be a  proper $C^{\infty}$ function on $M$. If $|\nabla f|^2\leq f$ on  $D_r$ for all $r$  and
\begin{equation}\Delta_f f+f
\leq k \label{V-1}
\end{equation}
for some constant $k$,
 then the weighted volume $\int_M e^{-f} dv <+\infty$ and
\begin{equation}V(r)\leq C r^{2k}  \label{V-2}
\end{equation}
for $r\geq 1$, where $C$ is a constant depending only on
$\int_M e^{-f} dv$.
\end{theorem}
\begin{proof} Since $f$ is proper, it is well defined  the integral $$I(t)=\frac1{t^k}\int_{\bar{D}_r}e^{-\frac{f}{t}}dv, \quad t>0.$$
\begin{equation}I'(t)=t^{-k-1}\int_{\bar{D}_r}(-k+\frac{f}{t})e^{-\frac{f}{t}}dv. \label{V-3}
\end{equation}
On the other hand,
\begin{align}
\int_{\bar{D}_r}\text{div}(e^{-\frac{f}{t}}\nabla f)dv=&\int_{\bar{D}_r}e^{-\frac ft}(\Delta f-\frac 1t|\nabla f|^2)dv\nonumber\\
\leq&\int_{\bar{D}_r}e^{-\frac ft}(|\nabla f|^2-f+k-\frac 1t|\nabla f|^2)dv\nonumber\\
=&\int_{\bar{D}_r}e^{-\frac ft}(\frac{t-1}{t}|\nabla f|^2-f+k)dv \nonumber\\
\leq &\int_{\bar{D}_r}e^{-\frac ft}(\frac{t-1}{t}f-f+k)dv, \quad t\geq 1\nonumber\\
=&\int_{\bar{D}_r}e^{-\frac ft}(-\frac{1}{t}f+k)dv,  \label{V-4}
\end{align}
Substituting (\ref{V-4}) into (\ref{V-3})  gives
$$I'(t)\leq -t^{-k-1}\int_{\bar{D}_r}\text{div}(e^{-\frac{f}{t}}\nabla
f)dv.$$
 At the regular value $r$ of $f$ and for $t\geq 1$, by
Stokes' Theorem, we have
\begin{align*}
I'(t)\leq &-t^{-k-1}\int_{\partial D_r}\left<e^{-\frac{f}{t}}\nabla f, \frac{\nabla f}{|\nabla f|}\right> dv\\
=&-t^{-k-1}\int_{\partial D_r}e^{-\frac{f}{t}}|\nabla f| dv\leq 0.
\end{align*}
Integrating $I'(t)$ over $t$  from $1$ to $r^2>1$, we get
\begin{equation}
r^{-2k}\int_{\bar{D}_r}e^{-\frac{f}{r^2}}dv\leq \int_{\bar{D}_r}e^{-f}dv. \label{V-5}
\end{equation}
Since the integrals in (\ref{V-5}) are upper semi-continuous for all
$r$, $(\ref{V-5})$ holds for all $r\geq1$. Note $2\sqrt{f}\leq r$
over $\bar{D}_r$. (\ref{V-5}) implies for all $r\geq 1$
\begin{align}
e^{-\frac14}r^{-2k}\int_{\bar{D}_r}dv&\leq r^{-2k}\int_{\bar{D}_r}e^{-\frac{f}{r^2}}dv\nonumber\\
&\leq \int_{\bar{D}_r}e^{-f}dv.\label{V-5-1}
\end{align}
Note that
\begin{align}
\int_{\bar{D}_r}e^{-f}dv-\int_{\bar{D}_{r-1}}e^{-f}dv&= \int_{\bar{D}_r\backslash \bar{D}_{r-1}}e^{-f}dv\nonumber\\
&\le e^{-\frac{(r-1)^2}{4}} \int_{\bar{D}_r\backslash \bar{D}_{r-1}}dv\nonumber\\
&\leq e^{-\frac{(r-1)^2}{4}}\int_{\bar{D}_r}dv.\label{V-6-1}
\end{align}
By (\ref{V-5-1}) and (\ref{V-6-1}),  we have
\begin{equation}\int_{\bar{D}_r}e^{-f}dv-\int_{\bar{D}_{r-1}}e^{-f}dv\leq e^{\frac 14}r^{2k}e^{-\frac{(r-1)^2}{4}}\int_{\bar{D}_r}e^{-f}dv\label{V-6}.
\end{equation}
We can choose a constant $r_0$ such that  $e^{\frac 14}r^{2k}e^{-\frac{(r-1)^2}{4}}<e^{-r}$,  for $r>r_0$. Then (\ref{V-6}) implies
 \begin{equation}\int_{\bar{D}_r}e^{-f}dv\leq \frac1{1-e^{-r}}\int_{\bar{D}_{r-1}}e^{-f}dv. \label{V-7}
 \end{equation}
 Then for any positive integer $N$, we have
 \begin{equation}\int_{\bar{D}_{r+N}}e^{-f}dv\leq \left( \prod_{i=0}^N\frac1{1-e^{-(r+i)}}\right)\int_{\bar{D}_{r-1}}e^{-f}dv<\infty. \label{V-8}
 \end{equation}
 This implies that $\int_Me^{-f}dv<+\infty$.
Moreover by (\ref{V-5-1}),   $$e^{-\frac14}r^{-2k}\int_{\bar{D}_r}dv\leq\int_{\bar{D}_r}e^{-f}dv\leq\int_Me^{-f}dv. $$
So for all $r\geq 1$,  $V(r)\leq Cr^{2k}.$
\end{proof}

\section{Application to self-shrinkers}
In this section,  we   discuss   the volume of self-shrinkers of  mean curvature flow by using Thoerem \ref{thm-f}.
Let $\Sigma$ be an  $n$-dimensional submanifold isometricly immersed  in $\mathbb{R}^{n+p}$. Let $A(V,W)=\overline{\nabla}_VW-\nabla_VW, V,W\in T\Sigma$ denote the second fundamental form and $\mathbf{H}$ denote the mean curvature vector of $\Sigma$ given by $\mathbf{H}=\sum_{i=1}^nA(e_i, e_i)$, where $\{e_i\}$ is a local orthonormal frame field of $\Sigma$.
$\Sigma$ is said to be a self-shrinker in $\mathbb{R}^{n+p}$ if it satisfies
\begin{equation}\label{P-8}
\mathbf{H}=-\frac{x^N}{2},
\end{equation}
where $x$ denotes the position vector, $(\cdot)^N$ denotes the orthogonal projection into the normal bundle $N\Sigma$ of $\Sigma$.

Let $f(x)=\frac{|x|^2}{4}$. Then  $\bar{\nabla} f=\frac x2,  |\bar{\nabla} f|^2=\frac{|x|^2}4$,
\begin{equation}
f-|\nabla f|^2=f- |\bar{\nabla} f|^2+|(\bar{\nabla}f)^N|^2=\frac{|x^N|^2}4=|{\bf H}|^2\geq 0, \label{S-2}
\end{equation}
\begin{equation}\label{S-3}
\begin{split}
\Delta_\Sigma f&=\frac{1}{4}\sum_{i=1}^n\bar{D}^2(|x|^2)(e_i, e_i)+\frac 14\langle {\bf H}, \bar\nabla |x|^2\rangle\\
&=\frac{n}{2}-\frac 14\langle x^N, x\rangle^2\\
&=\frac {n}{2}-f+|\nabla f|^2.
\end{split}
\end{equation}
Denote  by  $B_r(0)=\{x\in \mathbb{R}^{n+1}; |x|<r\}$ the ball in $\mathbb{R}^{n+1}$ of radius $r$ centered at  the origin.
\begin{theorem} \label{thm-S}  Let $\Sigma^n$ be a complete noncompact properly immersed self-shrinker  in Euclidian space $\mathbb{R}^{n+p}$ with nonnegative constant $\beta\leq\inf |{\bf H}|^2$.  Then,
there is a positive constant $C$ such that for $r\geq 1$
$$ \Vol( B_{r}(0)\cap \Sigma)\leq C r^{n-2\beta}.$$
\end{theorem}
\begin{proof}
By (\ref{S-2}),  $\inf \{f-|\nabla f|^2\}=\inf |{\bf H}|^2\geq\beta\geq 0$.  Take $\bar{f}=f-\beta.$ Then $$\bar{f}-|\bar{\nabla} f|^2=f-|\nabla f|^2-\beta\geq 0.$$
$$\Delta_\Sigma \bar{f}=\frac n2-\beta-\bar{f}+|\nabla\bar{f}|^2.$$
 The properness of the immersion of $\Sigma$ is the properness of $f$ on $\Sigma$ and hence $\bar{f}$ is proper.  The above implies $\bar{f}$ satisfies the assumption of Theorem \ref{thm-f}. Denote $$D_r=\{x\in \Sigma: 2\sqrt{\bar{f}}< r\}$$
 Then $$D_r=\{x\in \Sigma: |x|<  \sqrt{r^2+4\beta}\}.$$ By Theorem \ref{thm-f},
 $$\text{Vol}(D_r)\leq Cr^{n-2\beta},$$
 that is $$\text{Vol}( B_{ \sqrt{r^2+4\beta}}(0)\cap \Sigma)\leq C r^{n-2\beta}.$$
 So $$ \Vol( B_{r}(0)\cap \Sigma)\leq C r^{n-2\beta}.$$
 \end{proof}
In particular, if   $\Sigma^n$ be a self-shrinker in  $\mathbb{R}^{n+1}$. Since ${\bf H}=-H{\bf n}$,  where  $H=-\left<\nabla_{e_i}e_i,\bf{n}\right>$ and $\bf{n}$ denote the mean curvature and the outward unit normal respectively,  $\Sigma$ satisfies
\begin{equation}H=\frac{\left<x, {\bf n}\right>}{2}, \label{S-1}
\end{equation}
So we have Theorem \ref{thm-S-2-1}.
 \begin{remark}  Take cylinder self-shrinker  $S^{k}(\sqrt{2k})\times \mathbb{R}^{n-k}$,   $0\leq k\leq n$.  $H=\frac{\sqrt{k}}{\sqrt{2}},$  $n-2\beta=n-k$. Theorem \ref{thm-S} implies $ \Vol( B_{r}(0)\cap \Sigma)\leq C r^{n-k},$ which is sharp. The constant $C$ in theorem depends only on the integral $\int_\Sigma e^{-\frac{|x|^2}{4}}dv<\infty$ and $\beta$ explicitly as seen in the
proof. Take $\beta=0$ in Theorem \ref{thm-S}. We have the Euclidean growth upper bound proved by Ding and Xin \cite{DX}.
 \end{remark}

\section{Properness,  polynomial volume growth and  finite weighted volume for self-shrinkers}\label{S-5}

In this section, we discuss the equivalence of properness,
polynomial volume growth and  finite weighted volume for complete
self-shrinkers.  If $\Sigma$ is an $n$-dimensional submanifold in $\mathbb{R}^{n+p},  p\geq 1$,  $\Sigma$ is said to have polynomial volume growth  if there exist constants $C$ and $d$ so that for all $r\geq 1$
\begin{equation} \Vol (B_r(0)\cap \Sigma)\leq Cr^d. \label{P-1-1}
\end{equation}
When $d=n$ in (\ref{P-1-1}), $\Sigma$ is said to has Euclidean volume growth.
We get the following
\begin{theorem}\label{thm5}   For any complete $n$-dimensional immersed self-shrinker $\Sigma^n$ in $\mathbb{R}^{n+p}, p\geq 1$, the following statements are equivalent,
\begin{enumerate}
\item  [(i)]  $\Sigma$ is proper;
\item  [(ii)] $\Sigma$ has  Euclidean volume growth; 
\item  [(iii)] $\Sigma$ has  polynomial volume growth;
\item  [(iv)] $\Sigma$ has  finite weighted volume, that is, $\int_\Sigma e^{-\frac{|x|^2}{4}}dv<\infty.$
\end{enumerate}
\end{theorem}
\begin{proof}$(i)\Rightarrow (ii)$ has been proved in Theorem \ref{thm-S}. $(ii)\Rightarrow (iii)$ and $(iii)\Rightarrow (iv)$ are obvious.  So we only  need to
prove that $(iv)$ implies $(i)$. By contrary, if $\Sigma$ is not properly
immersed, there exists a number $R>0$ such that the pre-image
$E\subseteq \Sigma$ of $\bar{B}_R$ is not compact in $\Sigma$. Then for a positive constant $a>0$, there exists a
sequence $\{p_k\}$ of points in $E$ with $d_\Sigma(p_k,p_j)\geq a>0$ for
any $k\neq j$, where $d_\Sigma$ denotes
the intrinsic distance in $\Sigma$.  Denote by $B^{\Sigma}_{r}(p)$
the geodesic ball of $\Sigma$ of radius $r$  centered at  $p\in \Sigma$.  So
$B^{\Sigma}_{\frac a2}( p_k)\cap B^{\Sigma}_{\frac a2}( p_j)=\emptyset $ for any $k\neq j$.  Choose $a<2R$. Then $B^{\Sigma}_{\frac a2}( p_k)\subset B_{2R}$. Since
\begin{equation} {\bf H}=-\frac{x^N}{2}, \label{P-1}
\end{equation}
for any $p\in \Sigma\cap B_{2R}$,
\begin{equation}|{\bf H}|(p)\leq \frac{|x|}{2}\leq \frac {2R}{2}= R.\label{P-2}
\end{equation}
 We claim that there is a constant $C$ depending only on $R$ such that $\text{vol}(B^{\Sigma}_{\frac a2}( p_k))\geq Ca^n$ when $0<a\le\min \{2R, \frac{n}{2R}\}$ is suitably small.
 
  In fact for $p\in B^{\Sigma}_{\frac a2}( p_k))$, the extrinsic distance function $r_k(p):=\textrm{d}(p, p_k)$ from  $p_k$ satisfies
 \begin{align}
 \Delta_\Sigma r_k^2&=2n+2\left<{\bf{H}}, r_k\bar\nabla r_k\right>\nonumber \\
 &\ge 2n-2|{\bf H}|r_k\nonumber\\
 &\ge 2n-2r_kR. \label{P-3}
 \end{align}
 By $a\le \frac{n}{2R}$,  we have for $0<\mu\le \frac{a}2$
 \begin{align}
 \int_{B^{\Sigma}_{\mu}( p_k)} (2n-2Rr_k)dv&\le  \int_{B^{\Sigma}_{\mu}( p_k)}     \Delta_\Sigma r_k^2dv\nonumber\\
 &=  \int_{\partial B^{\Sigma}_{\mu}( p_k)}    \left< \nabla^{\Sigma} r_k^2, \nu\right >dv\nonumber\\
 &\le 2\mu A(\mu), \label{P-4}
 \end{align}
 where $\nu$ denotes the  outward unit normal vector of  $\partial B^{\Sigma}_{\mu}( p_k)$ and $A(\mu)$ denotes the area of $\partial B^{\Sigma}_{\mu}( p_k)$. Then using co-area formula in (\ref{P-4}), we get
\begin{equation}\int_0^\mu (n- Rs)A(s)ds\le \int_0^{\mu}\int_{d_{\Sigma}(p,p_k)=s}(n-Rr_k)d\sigma\le\mu A(\mu). \label{P-5}
\end{equation}
This implies
\begin{equation}(n-R\mu)V(\mu)\le \mu V'(\mu)\nonumber,
\end{equation}
where $V(\mu)$ denotes the volume of $B^{\Sigma}_{\mu}( p_k)$. So
\begin{equation}\frac{V'(\mu)}{V(\mu)}\geq \frac{n}{\mu}- R \label{P-6}
\end{equation}
Integrating (\ref{P-6}) from $\epsilon>0$ to $\mu$,  we have
\begin{equation}\frac{V(\mu)}{V(\epsilon)}\geq  (\frac{\mu}{\epsilon})^ne^{-R(\mu-\epsilon)}\nonumber
\end{equation}
Since
$\displaystyle\lim_{s\rightarrow 0^+}\frac{V(s)}{s^{n}}=\omega_{n}$,
\begin{equation}\label{P-7}V(\mu)\ge \omega_{n}\mu^ne^{- R\mu}.
\end{equation}
The  claim follows from (\ref{P-7}).
  Now  we have
\begin{equation}\label{eqn6}
\begin{split}
                   \int_\Sigma e^{-\frac{|x|^2}{4}}dv&\geq
\displaystyle\sum_{k=1}^{\infty}\int_{B^{\Sigma}_{\frac a2}( p_k))}e^{-\frac{|x|^2}{4}}dv\\
&\geq e^{-R^2}
 \displaystyle\sum_{k=1}^{\infty}\int_{B^{\Sigma}_{\frac a2}( p_k)}dv=+\infty.
                 \end{split}
 \end{equation}
It is a contradiction with the hypothesis of the finiteness of the weighted volume.
\end{proof}

Take $p=1$ in Theorem \ref{thm5}, we get Theorem \ref{thm4-1}.

\begin{remark} It was asked by Huai-Dong Cao if a complete
self-shrinker has polynomial volume growth. If there is any
example of nonproperly immersed self-shrinker, it must have
infinite weighted volume.  We do not know any such
examples.

\end{remark}

\section{Application to gradient shrinking Solitons}
In this section, we consider gradient shrinking solitons  of Ricci flows.
A  complete $n$-dimensional smooth
Riemannian   manifold $(M, g)$ is  called a
gradient shrinking Ricci soliton  if there exists a smooth
function $f$ on $M^n$ such that the Ricci tensor $R_{ij}$ of the
metric $g$ is given by
$$R_{ij}+\nabla_i\nabla_jf=\rho g_{ij}$$
for some positive constant $\rho>0$.  By
scaling $g_{ij}$ one can normalize $\rho=\frac{1}{2}$ so that
\begin{equation}R_{ij}+\nabla_i\nabla_jf=\frac{1}{2} g_{ij}.\label{G-1}
\end{equation}
Take trace in (\ref{G-1}), we have
$$R+\Delta f=\frac{n}{2}.$$
It  was proved by   Hamilton \cite{MR1375255}  that there is a
constant $C_0$ such that the scalar curvature $R$ and $f$ satisfy
\begin{equation}R+|\nabla f|^2-f=C_0.
\end{equation}
Since B.L. Chen  \cite{MR2520796} proved a complete gradient shrinking soliton has nonnegative scalar curvature (cf. Proposition 5.5 \cite{Cao09}), we can choose a new $f$,  still denoted by $f$,  such that $C_0$ is equal to the infimum of $R$.  Denote by $\beta=\inf R$. Hence for this $f$,  $0\leq |\nabla f|^2\leq f$ and thus $(M, g)$ is a gradient shrinking Ricci soliton with nonnegative potential function $f$ and
\begin{equation}
\Delta f-|\nabla f|^2+f=\frac{n}{2}-\beta.\label{G-2}
\end{equation}
 Using Theorem \ref{thm-f}, we get the following

\begin{lemma}\label{lem} Let $(M^n, g, f)$ be a complete noncompact gradient shrinking solition with $f$ satisfying (\ref{G-1}) and (\ref{G-2}) with $\beta=\inf R$. Then
$\int_M e^{-f} dv<\infty$ and
$$V(r)\leq C r^{n-2\beta}$$
for $r\geq 1$, where $C$ is a constant depending only on
$\int_M e^{-f} dv$.
\end{lemma}
\begin{proof} It suffices to show the properness of $f$. Take $\bar{f}=f+\beta$. In \cite{MR2732975},  it is proved that
the
potential function $\bar{f}$ satisfies for $r$ sufficiently large
\begin{equation}\frac{1}{4} (r(x)-c_1)^2\leq \bar{f}(x)\leq \frac{1}{4} (r(x)+c_2)^2,\label{G-2-2}
\end{equation}
where $r(x)=d(x_0, x)$ is the distance function from some fixed
point $x_0\in M$, $c_1$ and $c_2$ are positive constants depending
only on $n$ and the geometry of $g_{ij}$ on the unit ball
$B_{x_0}(1)$. Hence $\bar{f}$ is proper and thus $f$ is also.
\end{proof}
\begin{theorem}\label{thm-shrinking} (Thm.\ref{thm-shrinking-1}) Let $(M^n, g, f)$ be a complete noncompact gradient shrinking solition with $f$ satisfying (\ref{G-1}) and $\beta=\inf R$. Then
 there exists  a  constant $C>0$ such that
$$ \Vol(B_{r}(x_0))\leq C r^{n-2\beta}$$
for $r\geq 1$, where  $C$ is a constant depending on $x_0$ and $ \Vol(B_{r}(x_0))$ denotes the volume of the geodesic ball of $M$ of radius $r$ centered at $x_0\in M$.
\end{theorem}
\begin{proof} Take $\tilde{f}=f+\beta-C_0$. Then $\tilde{f}$ satisfies Lemma \ref{lem}.  So let $D_r=\{x\in M: 2\sqrt{\tilde{f}}< r\}$. Then $$V(D_r)\leq Cr^{n-2\beta}$$ for $r\geq 1.$  By (\ref{G-2-2}), it is easy to know $B_{r}(x_0)\subseteq D_{r+c}$ for some constant $c$ and sufficiently large $r$.
\end{proof}
\begin{remark}  The volume upper bounds of the level set $D_r$ and the geodesic ball  $B_r(x_0)$ in Theorem \ref{thm-shrinking} are optimal. For example  the Gaussian shrinker, namely the flat
Euclidean space $(\mathbb R^n, g_0)$ with the potential function
$f=|x|^2/4$ and cylinder shrinking solitons $(S^{n-k}\times \mathbb{R}^k, g)$,   $k\geq 1$,  where $g=2(n-k-1)g_{S^{n-k}}+g_{\mathbb{R}^k}$ and $f (\theta, x)=\frac{|x|^2}{4},  \theta\in S^{n-k}, x\in \mathbb{R}^k.$ 
\end{remark}
\begin{remark}Since $\beta=\inf R\geq 0$, Theorem \ref{thm-shrinking} implies the Euclidean volume growth estimate of Cao and the second author \cite{MR2732975}. The method of the proof presented in Lemma \ref{lem} is different .
\begin{remark}By \cite{MR2732975}  the equation of shrinking Ricci soliton gives  the estimate of potential function, we can use the finiteness of weighted volume to conclude   that   the
fundamental group $\pi_1(\Sigma)$ of $\Sigma$ is finite, which was proved by W. Wylie \cite{MR2373611}. But for a self-shrinker, the potential function comes from the restriction of metric  of ambient space. We cannot use it to study the fundamental group of  self-shrinker.
\end{remark}

\end{remark}


\begin{bibdiv}
\begin{biblist}

\bib{Cao09}{article}{
   author={Cao, Huai-Dong},
   title={Geometry of complete gradient shrinking Ricci solitons},
   conference={
      title={Geometry and analysis. No. 1},
   },
   book={
      series={Adv. Lect. Math. (ALM)},
      volume={17},
      publisher={Int. Press, Somerville, MA},
   },
   date={2011},
   pages={227--246},
   review={\MR{2882424}},
}
		

\bib{MR2648937}{article}{
   author={Cao, Huai-Dong},
   title={Recent progress on Ricci solitons},
   conference={
      title={Recent advances in geometric analysis},
   },
   book={
      series={Adv. Lect. Math. (ALM)},
      volume={11},
      publisher={Int. Press, Somerville, MA},
   },
   date={2010},
   pages={1--38},
   review={\MR{2648937 (2011d:53061)}},
}

\bib{MR2488948}{article}{
   author={Cao, Huai-Dong},
   author={Chen, Bing-Long},
   author={Zhu, Xi-Ping},
   title={Recent developments on Hamilton's Ricci flow},
   conference={
      title={Surveys in differential geometry. Vol. XII. Geometric flows},
   },
   book={
      series={Surv. Differ. Geom.},
      volume={12},
      publisher={Int. Press, Somerville, MA},
   },
   date={2008},
   pages={47--112},
   review={\MR{2488948 (2010c:53096)}},
}


\bib{MR2732975}{article}{
   author={Cao, Huai-Dong},
   author={Zhou, Detang},
   title={On complete gradient shrinking Ricci solitons},
   journal={J. Differential Geom.},
   volume={85},
   date={2010},
   number={2},
   pages={175--185},
   issn={0022-040X},
   review={\MR{2732975}},
}


\bib{MR2520796}{article}{
   author={Chen, Bing-Long},
   title={Strong uniqueness of the Ricci flow},
   journal={J. Differential Geom.},
   volume={82},
   date={2009},
   number={2},
   pages={363--382},
   issn={0022-040X},
   review={\MR{2520796 (2010h:53095)}},
}



\bib{CM1}{article}{
   author={Colding, Tobias H.},
   author={Minicozzi II, William P.},
   title={Generic mean curvature flow I; generic singularities},
 journal={Ann. of Math.},
   volume={175},
   date={2012},
   number={2},
   pages={755--833},
   issn={},
   review={},
}

\bib{CM2}{article}{
   author={Colding, Tobias H.},
   author={Minicozzi II, William P.},
   title={Smooth compactness of self-shrinkers},
 journal={Comment. Math. Helv.},
   volume={87},
   date={2012},
   number={1},
   pages={463--475},
   issn={},
   review={},
}

\bib{DX}{unpublished}{
   author={Ding, Qi},
   author={Xin, Yuanlong},
   title={Volume growth, eigenvalue and compactness for self-shrinkers},
 note={ arXiv:1101.1411.},
   }

\bib{MR1025164}{article}{
   author={Ecker, Klaus},
   author={Huisken, Gerhard},
   title={Mean curvature evolution of entire graphs},
   journal={Ann. of Math. (2)},
   volume={130},
   date={1989},
   number={3},
   pages={453--471},
   issn={0003-486X},
   review={\MR{1025164 (91c:53006)}},
   doi={10.2307/1971452},
}


\bib{MR1375255}{article}{
   author={Hamilton, Richard S.},
   title={The formation of singularities in the Ricci flow},
   conference={
      title={Surveys in differential geometry, Vol.\ II},
      address={Cambridge, MA},
      date={1993},
   },
   book={
      publisher={Int. Press, Cambridge, MA},
   },
   date={1995},
   pages={7--136},
   review={\MR{1375255 (97e:53075)}},
}

\bib{Hu}{article}{
   author={Huisken, Gerhard},
   title={Asymptotic behavior for singularities of the mean curvature flow},
   journal={J. Differential Geom.},
   volume={31},
   date={1990},
   number={1},
   pages={285--299},
   issn={0022-040X},
   review={\MR{1030675 (90m:53016)}},
}

\bib{P}{unpublished}{
   author={Perelman, G.},
   title={The entropy formula for the Ricci flow and its geometric applications},
  note={ arXiv: math.DG/0211159.},

}

\bib{MR2373611}{article}{
   author={Wylie, William},
   title={Complete shrinking Ricci solitons have finite fundamental group},
   journal={Proc. Amer. Math. Soc.},
   volume={136},
   date={2008},
   number={5},
   pages={1803--1806},
   issn={},
   review={MR2373611 (2009a:53059)},
}


\bib{Zh}{article}{
   author={Zhang, Shi Jin},
   title={On a sharp volume estimate for gradient Ricci solitons with scalar
   curvature bounded below},
   journal={Acta Math. Sin. (Engl. Ser.)},
   volume={27},
   date={2011},
   number={5},
   pages={871--882},
   issn={1439-8516},
   review={\MR{2786449 (2012g:53081)}},
}



   \end{biblist}
\end{bibdiv}
\end{document}